\documentclass[12pt,reqno]{article}

\usepackage[usenames]{color}
\usepackage{amssymb}
\usepackage{amsmath}
\usepackage{amsthm}
\usepackage{amsfonts}
\usepackage{amscd}
\usepackage{graphicx}

\usepackage[colorlinks=true,
linkcolor=webgreen,
filecolor=webbrown,
citecolor=webgreen]{hyperref}

\definecolor{webgreen}{rgb}{0,.5,0}
\definecolor{webbrown}{rgb}{.6,0,0}

\usepackage{color}
\usepackage{fullpage}
\usepackage{float}

\usepackage{graphics}
\usepackage{latexsym}
\usepackage{epsf}
\usepackage{breakurl}

\newcommand{\floor}[1]{\left\lfloor #1\right\rfloor}

\begin{document}

\theoremstyle{plain}
\newtheorem{theorem}{Theorem}
\newtheorem{corollary}[theorem]{Corollary}
\newtheorem{lemma}[theorem]{Lemma}
\newtheorem{proposition}[theorem]{Proposition}

\theoremstyle{definition}
\newtheorem{definition}[theorem]{Definition}
\newtheorem{example}[theorem]{Example}
\newtheorem{conjecture}[theorem]{Conjecture}

\theoremstyle{remark}
\newtheorem{remark}[theorem]{Remark}

\begin{center}
\vskip 1cm{\LARGE\bf 
There are no Collatz $m$-Cycles with $m\leq 91$
}
\vskip 1cm
\large
Christian Hercher\\
Institut f\"{u}r Mathematik\\
Europa-Universit\"{a}t Flensburg\\
Auf dem Campius 1b\\
24944 Flensburg\\
Germany \\
\href{mailto:christian.hercher@uni-flensburg.de}{\tt christian.hercher@uni-flensburg.de} \\
\end{center}

\vskip .2 in

\begin{abstract}
The Collatz conjecture (or ``Syracuse problem'') considers
recursively-defined sequences of positive integers where $n$ is
succeeded by $\tfrac{n}{2}$, if $n$ is even, or $\tfrac{3n+1}{2}$, if
$n$ is odd. The conjecture states that for all starting values $n$ the
sequence eventually reaches the trivial cycle $1, 2, 1, 2, \ldots$
. We are interested in the existence of nontrivial cycles.

Let $m$ be the number of local minima in such a nontrivial cycle. Simons
and de Weger proved that $m \geq 76$. With newer bounds on the range of
starting values for which the Collatz conjecture has been checked, one
gets $m \geq 83$. In this paper, we prove $m \geq 92$.

The last part of this paper considers what must be proven in order to
raise the number of odd members a nontrivial cycle has to have to the
next bound---that is, to at least $K\geq1.375\cdot 10^{11}$. We prove
that it suffices to show that, for every integer smaller than or equal to
$1536\cdot2^{60}=3\cdot2^{69}$, the respective Collatz sequence enters
the trivial cycle. This reduces the range of numbers to be checked by
nearly $60$\%.
\end{abstract}

\section{Introduction}

The Collatz conjecture\footnote{Named after German mathematician Lothar
Collatz, 1910--1990.} (or Syracuse problem) considers recursively-defined
sequences of positive integers.

\begin{definition}
Let $n$ be a positive integer. The Collatz operator $C: \mathbb{Z}_{>0}
\rightarrow \mathbb{Z}_{>0}$ is defined as follows: \[C(n):=\begin{cases}
\frac{n}{2}, & \text{ if $n$ is even;}\\ \frac{3n+1}{2}, & \text{ if
$n$ is odd.}\end{cases}\] We define an \textit{e-step} to be of type
$n\mapsto \tfrac{n}{2}$, since it occurs after an even number~$n$, and
an \textit{o-step} to be of type $n\mapsto \tfrac{3n+1}{2}$, since it
follows after an odd number~$n$.  \end{definition}

\begin{conjecture}[Collatz]
For all $n\in\mathbb{Z}_{>0}$, the sequence $n, C(n), C^2(n):=C(C(n)), C^3(n), \ldots$ eventually reaches the trivial cycle $1, 2, 1, 2, \ldots$ .
\end{conjecture}

\begin{remark}
While the Collatz conjecture is
currently unproven, there are good reasons to believe it holds. In \cite{RT}, for example, Terras showed that the respective Collatz sequence enters the trivial cycle for almost all integers. More recently, Tao also proved in \cite{Tao} that, for every function $f:\mathbb{Z}_{>0}\rightarrow \mathbb{R}$ with $\lim\limits_{n\rightarrow\infty} f(n)= \infty$ and almost every positive integer~$n$, the Collatz sequence starting with $n$ reaches a number less than or equal to $f(n)$:
\[\liminf_{k\rightarrow \infty} C^k(n)\leq f(n).\]
Nevertheless, the conjecture could be wrong in two possible ways:
\begin{itemize}
\item There could be an unbounded sequence $n, C(n), \ldots$ with $\limsup_{i\rightarrow \infty} C^i(n)=\infty$.
\item There could be a nontrivial cycle $n, C(n), \ldots$ with $C^i(n)=n$ for some positive integer~$i$.
\end{itemize}
This paper concerns itself with the second case. 
Specifically, we consider the properties a nontrivial cycle would need to have.
\end{remark}

Over time, different projects, \cite{LV,Roosendaal,TOS,TOS2,He,He2,Barina}, have taken various approaches to prove or disprove the Collatz conjecture by checking whether the respective sequences of increasing numbers converge to the trivial cycle. The best known bound is from Barina~\cite{Barina-www}:

\begin{definition}\label{Definition8}
We define $X_0$ as the largest number for which it is known that for all positive integers $n\leq X_0$ their Collatz sequences reach the trivial cycle. As of the date of writing this paper, we have
\[X_0=704 \cdot 2^{60}.\]
\end{definition}

Since the Collatz sequences converge to the trivial cycle for all numbers
less than or equal to $X_0$,  we can derive estimates of properties of
hypothetically existing nontrivial cycles. On the basis of the 
value of $X_0$ known at the time,  Eliahou~\cite{SE} used methods derived from
diophantine approximation and continued fractions to prove that such a
nontrivial cycle must include at least $K>10^7$ o-steps. This bound was
subsequently improved by Hercher and Puchert~\cite{He,He2,Pu} to the
current  $K>7.2\cdot 10^{10}$ .

Another property of such a nontrivial cycle was investigated by Simons
and de Weger~\cite{SW}. As $\tfrac{3n+1}{2}>n>\tfrac{n}{2}$, such a
nontrivial cycle has strictly increasing and strictly decreasing passages:
\begin{definition}
Let $n, C(n), \ldots, C^i(n)=n$ be a nontrivial cycle. (That is, $n>2$.) Let us call it an $m$-cycle if it has exactly $m$ local minima (and therefore also exactly $m$ local maxima). 
\end{definition}

Simons and de Weger~\cite{SW} proved that $m\geq 68$ for
every nontrivial cycle. Based on the rate of growth of $X_0$ in the year
2005, when \cite{SW} appeared, they estimated that bounds on $m$ like
$m\geq 77$ were unlikely to be obtained before the
year 2419. In an updated version of the paper~\cite{SW-www} in 2010,
the same authors used what was then
a new bound $X_0$ to prove that $m \geq 76$. Regarding future bounds,
estimates suggested that $m\geq 83$ could not be reached before the year
2429 if progress on $X_0$ maintained the then-current pace and no new
insights into the topic emerged. In 2018, Hercher~\cite{He,He2}---using
his newly derived bound on $X_0$---proved that $m \geq 77$.

With the current best bound on $X_0$, Simons and de Weger
\cite{SW,SW-www} suggested that with methods of their paper $m \geq 83$
could be proven.\footnote{Private e-mail conversation in September 2021.}
In this paper, we prove that $m \geq 92$.

\subsection{Outline}
In Section~\ref{section3}, we consider sums of the form $T(n_i)=\sum_{k=0}^{k_i} \tfrac{1}{C^k(n_i)}$. This is the sum of the reciprocals of the odd numbers starting with and following a local minimum $n_i$ in an $m$-cycle. In Remark~\ref{Remark_T-trvialBound}, Lemma~\ref{Lemma_T-Abschaetzung}, and Corollary~\ref{Korrolar_besserTAbschaetzung}, we give subsequently improving upper bounds on the average of these $T(n_i)$ in sets of up to three consecutive such terms. This leads to Lemma~\ref{LemmaT_Abschaetzung} and Theorem~\ref{Theorem_bessereT_Summe}, where we give bounds on sums of an arbitrary number of consecutive $T(n_i)$.

This is motivated and used in Section~\ref{Section4}. At the start of this section, some results on nontrivial cycles in Collatz sequences are listed and proved in a way similar to that in~\cite{SE,SW,He}. In particular, Theorem~\ref{Theorem11} gives bounds on the fraction $\tfrac{K+L}{K}$ in terms of $\sum_{i=1}^m T(n_i)$. Here, $K$ is the number of odd and $L$ the number of even entries in the cycle and therefore the number of o-steps and e-steps, respectively. Together with the bounds on such sums over $T(n_i)$ and further refinements, this leads to better bounds on $\tfrac{K+L}{K}$. These bounds are given in Corollary~\ref{Corollary_TSumme_ohneRestglied} and Theorem~\ref{Theorem15}. (Corollary~\ref{Computer_estimate_m_T} is a version of Corollary~\ref{Corollary_TSumme_ohneRestglied} which gives an absolute bound independently of $X_0$, where we used a computer program to derive it.) 

The proof of Theorem~\ref{Theorem15} is based on the idea that $k_i$ cannot grow at an unlimited rate: There is a constant $1<\delta<2$ such that $k_{i+j}\leq \delta^j \cdot \tfrac{\log (n_i+1)}{\log (2)}$ for all~$i, j$; see Lemma~\ref{Lemma14}. Then we split the sum $\sum_{i=1}^m T(n_i)$ into two parts. The first one, with ``small'' $n_i$, is bounded using Theorem~\ref{Theorem_bessereT_Summe} from Section~\ref{section3}, and the second one with ``large'' and ``growing'' $k_i$ (and therefore ``large'' $n_i$, too). The second partial sum is estimated in the proof of Theorem~\ref{Theorem15} itself.

The results for the bounds  on $\tfrac{K+L}{K}$ can then be used together with Lemma~\ref{Lemma16}. This lemma states how to obtain the smallest possible denominator of all fractions in a given nonempty open interval. We use this lemma in our main theorem, Theorem~\ref{Theorem17}, to show that there is no $m$-cycle with $m\leq 91$. To do so, we assume $m\leq 91$ and give a new lower bound on $K$, which then can be used in the same way to get an even better lower bound, and so on. Finally, we arrive at a proven lower bound on $K$ which is larger than the upper bound for this value of~$m$ given by Simons and de Weger~\cite{SW-www}. This result shows that no such $m$-cycle with $m\leq 91$ can exist.

In the same way, we can give new lower bounds on $K$ for $m$-cycles with $m\geq 92$. A selection of them is given in Corollary~\ref{Corollary_Computer_greater_m_estimate}. The computations for Theorem~\ref{Theorem17} and Corollary~\ref{Corollary_Computer_greater_m_estimate} were done using a {\tt Sagemath} worksheet, with the total computing time lasting a few minutes.

Finally, in Section~\ref{section5} the methods of the previous sections are used to consider the case of nontrivial Collatz cycles independent of~$m$. The result on the average value of the $T(n_i)$ yielded Lemma~\ref{Lemma_T-Abschaetzung_ohne_m} is analogous to the ones given by Lemma~\ref{Lemma_T-Abschaetzung} and Corollary~\ref{Korrolar_besserTAbschaetzung}. However there the values $T(n_i)$ are weighted by the numbers~$k_i$. This is used to get Theorem~\ref{Theorem-Abschaetzungohnem}, which yields a result analogous to that of Corollary~\ref{Corollary_TSumme_ohneRestglied}, but independent of~$m$. As stated in Remark~\ref{RemarkAbschaetzungohnem}, this lowers the value $X_0$ has to have for proving $K>1.375\cdot 10^{11}$ (which is the next threshold reachable) from $3781\cdot2^{60}$ to $2836\cdot 2^{60}$. To get this result with an even smaller known value of $X_0$, we implemented the methods in a computer program which---after extensive computations---proved Corollary~\ref{Corollary_Computer-Bound}, which states that already $X_0\geq 1536\cdot 2^{60}=3\cdot 2^{69}$ is sufficient.
 
\section{Upper bounds on $T(n_i)$} \label{section3}

\begin{definition}\label{Definition_k_i}
For a given $m$-cycle, we call the local minima in order of appearance in the cycle $n_1, n_2, \ldots, n_m$. To keep the notation simple, we further set $n_{m+1}:=n_1, n_{m+2}:=n_2$ and so on. Note that we do not need $n_1$ to be the smallest of these minima, as every cyclic permutation of these minima yields to the same results. All we need is that, for all $i$, the next local minimum after $n_i$ in this Collatz trajectory is $n_{i+1}$. 

Further let $k_i$ be the exact number of consecutive o-steps directly following the local minimum~$n_i$, and let  $\ell_i$  be the exact number of e-steps directly following them. Thus, the numbers $n_i, C(n_i), \ldots$, and $C^{k_i-1}(n_i)$ are all odd (as they are followed by an o-step) and $C^{k_i}(n_i)$ is even. Therefore, $C^{k_i}(n_i)$ is the local maximum following $n_i$. From that point on, the numbers $C^{k_i}(n_i), C^{k_i+1}(n_i), \ldots$ and $C^{k_i+\ell_i-1}(n_i)$ are all even (since they are followed by an e-step) and $C^{k_i+\ell_i}(n_i)$ is odd. Therefore, $C^{k_i+\ell_i}(n_i)=n_{i+1}$ is the next local minimum after $n_i$.

For positive integers $i$, let 
\[T(n_i):=\sum_{t=0}^{k_i-1} \frac{1}{C^t(n_i)}.\] 
\end{definition}

\begin{remark}\label{Remark_T-trvialBound}
We have $T(n_i)<\frac{3}{n_i}<\tfrac{3}{X_0}$ for all $1\leq i$.
\end{remark}

\begin{proof}
As $C(n)>\tfrac{3}{2} n$ for odd $n$, we have $C^t(n_i)\geq \left(\tfrac{3}{2}\right)^t \cdot n_i$, and therefore
\[
T(n_i)=\sum_{t=0}^{k_i-1} \frac{1}{C^t(n_i)} < \sum_{t=0}^{\infty} \frac{1}{n_i \cdot \left(\tfrac{3}{2}\right)^t} = \frac{1}{n_i} \cdot \sum_{t=0}^{\infty} \left(\frac{2}{3}\right)^t = \frac{3}{n_i} < \frac{3}{X_0}.
\] 
The last inequality follows from $n_i>X_0$.
\end{proof}

\begin{lemma}\label{Lemma12}
Let $n$ and $k$ be positive integers, where $n, C(n), \ldots, C^{k-1}(n)$ are all odd. Then $n\equiv -1 \pmod{2^k}$. In particular, $n\geq 2^k-1$.
\end{lemma}

\begin{proof}
We proceed by induction on~$k$ to show that not only $n\equiv -1 \pmod{2^k}$, but also $C^k(n)=a \cdot 3^k-1$, if $n=a\cdot 2^k-1$. If $k=1$, then the statement is obvious: We have $n=a\cdot 2^1 -1$ for some positive integer~$a$ and $C^1(n)=a\cdot 3^1 - 1$. 

Now let $k$ be at least~2 and $n, C(n), \ldots, C^{k-2}(n)$ all odd. Then there is a positive integer~$a$ with $n=a\cdot 2^{k-1} -1$ and $C^{k-1}(n)=a\cdot 3^{k-1} -1$. Since in order for $C^{k-1}(n)$ to be odd, $a$ must be even, $a=2a^{\prime}, n=a^{\prime} \cdot 2^k-1$ and $C^k(n)=a^{\prime} \cdot 3^k-1$. Thus, $n\equiv -1\pmod{2^k}$ and $n\geq 2^k-1$.
\end{proof}

Note that this well-known statement has been used before and proven in many publications on this subject, e.g., \cite{RT}.
\begin{lemma}\label{Collatz-merger}
Let $n_i$ be an odd positive integer. Let $k_i$ be the exact number of o-steps directly following $n_i$ in its Collatz sequence and $\ell_i$ the exact number of e-steps following them.

If $\ell_i\geq 2$, then the Collatz sequences of $n_i$ and $n_i^{\prime}:=\tfrac{n_i-1}{2}$ merge since $C^{k_i+2}(n_i)=C^{k_i+1}(n_i^{\prime})$.
\end{lemma}

\begin{proof}
 Let $n_i, k_i$ and $\ell_i$ be as in the lemma. Then, there is a positive integer~$a$ with $n_i=a \cdot 2^{k_i+2}+2^{k_i}-1$ if $k_i$ is even, or $n_i=a \cdot 2^{k_i+2}+3\cdot 2^{k_i}-1$ if $k_i$ is odd. To see this, we recall with Lemma~\ref{Lemma12} that $n_i$ is congruent to $-1$ modulo $2^{k_i}$, but not modulo $2^{k_i+1}$. Therefore, there is a positive integer $b$ with $n_i=b \cdot 2^{k_i+1} + 2^{k_i}-1$. With this, according to the proof
of Lemma~\ref{Lemma12} we get 
\[C^{k_i+1}(n_i)=b \cdot 3^{k_i} + \frac{3^{k_i}-1}{2}=\frac{(2b+1) \cdot 3^{k_i}-1}{2}.\]
Since we want another e-step to follow this number, it has to be even. Hence, we get $(2b+1) \cdot 3^{k_i}\equiv 1\pmod{4}$. For even $k_i$ we have $3^{k_i}\equiv 1\pmod{4}$, thus, $b=2a$ has to be even. For odd $k_i$ we have $3^{k_i}\equiv 3\pmod{4}$, such that $b=2a+1$ has to be odd. 

If $k_i$ is even, this leads to the equality $C^{k_i+2}(n_i)=a\cdot 3^{k_i} + \tfrac{3^{k_i}-1}{4}$. But this number is also a Collatz successor of the number $n_i^{\prime}:=\tfrac{n_i-1}{2}=a\cdot 2^{k_i+1} + 2^{k_i-1}-1$, since through $n_i^{\prime} \equiv 2^{k_i-1}-1 \pmod{2^{k_i}}$ we get $C^{k_i}(n_i^{\prime}) = a \cdot 2 \cdot 3^{k_i-1} + \tfrac{3^{k_i-1}-1}{2}$. But when $k_i$ is even we know that $\tfrac{3^{k_i-1}-1}{2}$ is odd and that, therefore, $C^{k_i}(n_i^{\prime})$ is odd, too. As a result, we get 
\[C^{k_i+1}(n_i^{\prime})=a\cdot 3^{k_i}+\frac{3^{k_i}-1}{4}=C^{k_i+2}(n_i).\]

Similarly, for odd $k_i$ we get 
\begin{align*}
C^{k_i+1}(n_i)&=\frac{(2(2a+1)+1) \cdot 3^{k_i}-1}{2}=\frac{4a\cdot 3^{k_i} +3 \cdot 3^{k_i} -1}{2}\\
\quad&=2a\cdot 3^{k_i} + \frac{3^{k_i+1}-1}{2}
\end{align*}
and therefore $C^{k_i+2}(n_i)=a\cdot 3^{k_i} + \tfrac{3^{k_i+1}-1}{4}$. Now this is $C^{k_i+1}(n_i^{\prime})$ with $n_i^{\prime}=\tfrac{n_i-1}{2}=a \cdot 2^{k_i+1} + 3\cdot 2^{k_i-1}-1$, too: By $n_i^{\prime}\equiv 2^{k_i-1}-1\pmod{2^{k_i}}$, we know $C^{k_i}(n_i^{\prime})=a \cdot 2 \cdot 3^{k_i-1}+\tfrac{3^{k_i-1}-1}{2}$. As $k_i$ is odd, $C^{k_i}(n_i^{\prime})$ is odd, too, and we get 
\[C^{k_i+1}(n_i^{\prime})=a\cdot 3^{k_i}+\frac{3^{k_i}-1}{4}=C^{k_i+2}(n_i).\]
\end{proof}

\begin{remark}
 From this lemma we also get a small result for unbounded Collatz sequences:
 
 Let $n_1$ be the smallest positive integer, if any, for which the Collatz sequence is unbounded. Let $k_1$ and $\ell_1$ be as above. Then $\ell_1=1$ and $k_1\geq 2$. (If $\ell_1$ were greater, then there would be a smaller number $n_1^{\prime}$ which would also have an unbounded Collatz sequence, and this would contradict the minimality of $n_1$. The same would hold, if $k_1=\ell_1=1$, since in that case we would have $C^2(n_1)=\tfrac{3n_1+1}{4}<n_1$, which also has an unbounded sequence.)
\end{remark}

\begin{lemma}\label{Lemma_T-Abschaetzung}
For all $1\leq i$, we have 
\begin{align*}
T(n_i)&< \frac{35}{18} \cdot \frac{1}{X_0} \text{ or}\\ 
T(n_i)+T(n_{i+1})&<2 \cdot \frac{35}{18} \cdot \frac{1}{X_0}.
\end{align*}  
\end{lemma}

\begin{proof}
Let $k_i$ be as in Definition~\ref{Definition_k_i}. For $k_i\leq 2$ we have 
\[T(n_i) < \frac{1}{n_i} \cdot \left( 1+\frac{2}{3} \right) =\frac{5}{3} \cdot \frac{1}{n_i} < \frac{35}{18} \cdot \frac{1}{n_i}<\frac{35}{18} \cdot \frac{1}{X_0}.\] 

Now let $k_i\geq 3$. Then 
\begin{align*}
T(n_i)&=\sum_{t=0}^{k_i-1} \frac{1}{C^t(n_i)}< \sum_{t=0}^{k_i-1} \frac{1}{n_i \cdot \left(\tfrac{3}{2}\right)^t}\\
 &= \frac{1}{n_i} \cdot \frac{1-\left(\frac{2}{3}\right)^{k_i}}{1-\frac{2}{3}}\\
 &= \frac{1}{n_i} \cdot  \left(3 - 3 \cdot \left(\frac{2}{3}\right)^{k_i} \right)
\end{align*}
and $C^{k_i}(n_i) > \left(\tfrac{3}{2}\right)^{k_i} \cdot n_i$. As an e-step follows, we get 
\[C^{k_i+1}(n_i) > \frac{1}{2} \cdot \left(\tfrac{3}{2}\right)^{k_i} \cdot n_i.\]
If $C^{k_i+1}(n_i)$ is already the next local minimum, that is $n_{i+1}=C^{k_i+1}(n_i)$, we get 
\begin{align*}
T(n_{i+1})&< \frac{3}{n_{i+1}} < 6 \cdot \left(\frac{2}{3}\right)^{k_i} \cdot \frac{1}{n_i}, \\
T(n_i)+T(n_{i+1}) &< \frac{1}{n_i} \cdot \left(3 - 3 \cdot \left(\frac{2}{3}\right)^{k_i} + 6 \cdot \left(\frac{2}{3}\right)^{k_i}  \right)\\
 &=  \frac{1}{n_i} \cdot \left(3 + 3 \cdot \left(\frac{2}{3}\right)^{k_i}  \right)\\
 &\leq \frac{1}{n_i} \cdot \frac{35}{9} < 2 \cdot \frac{35}{18} \cdot \frac{1}{X_0}.
\end{align*}
The second-to-last inequality follows from $\left(\tfrac{2}{3}\right)^{k_i}\leq \tfrac{8}{27}$ for $k_i\geq 3$.

Finally, we must consider the case when there are at least two e-steps after the $k_i$ o-steps following $n_i$. Then from Lemma~\ref{Collatz-merger} we know that  there exists an integer $n_i^{\prime}=\tfrac{n_i-1}{2}$ for which the Collatz sequence merges with that of $n_i$, and thus with the considered cycle.

But $n_i^{\prime}>X_0$ since for every positive integer less than or equal to $X_0$ the Collatz sequence enters the trivial cycle. Hence, we have $n_i>2\cdot n_i^{\prime}>2X_0$ and thus $T(n_i)<\tfrac{3}{n_i}<\tfrac{3}{2} \cdot \tfrac{1}{X_0}<\tfrac{35}{18} \cdot \tfrac{1}{X_0}$, which concludes the proof. 
\end{proof}

\begin{lemma}\label{LemmaT_Abschaetzung}
Let $0\leq m_1$ be an integer. Then
\[ \sum_{i=1}^{m_1} T(n_i)<\frac{35m_1+19}{18} \cdot \frac{1}{X_0}.\]
\end{lemma}

\begin{proof}
We split the sum into single summands or sums of two consecutive summands, to each of which we apply the previous lemma. From Lemma~\ref{Lemma_T-Abschaetzung}, we know that $T(n_1)<\tfrac{35}{18} \cdot \tfrac{1}{X_0}$ or $T(n_1)+T(n_2) < 2\cdot \frac{35}{18} \cdot \tfrac{1}{X_0}$. In the first case, our first part consists solely of $T(n_1)$. In the second case, it is the partial sum $T(n_1)+T(n_2)$. The second partial sum is defined in a similar way, and so on. We say such a partial sum is complete if from Lemma~\ref{Lemma_T-Abschaetzung} we know that it is smaller than $\tfrac{35}{18} \cdot \tfrac{1}{X_0}$ for one summand or $2\cdot \tfrac{35}{18} \cdot \tfrac{1}{X_0}$ for two summands. Since all such partial sums until the one with $T(n_{m_1-1})$ are complete, they therefore have an average value of less than $\tfrac{35}{18} \cdot \tfrac{1}{X_0}$ per summand. Thus, if $T(n_{m_1})$ is in a complete partial sum we get $\sum_{i=1}^{m_1} T(n_i)<\frac{35}{18} \cdot m_1 \cdot \tfrac{1}{X_0}$, and if $T(n_{m_1})$ is not part of a complete partial sum, it cannot be in a partial sum together with $T(n_{m_1-1})$. Thus, we have
\begin{align*} 
\sum_{i=1}^{m_1} T(n_i)&= T(n_{m_1}) + \sum_{i=1}^{m_1-1} T(n_i)< T(n_{m_1}) + (m_1-1) \cdot \frac{35}{18} \cdot \frac{1}{X_0} \\
\quad &< \frac{3}{X_0} + (m_1-1) \cdot \frac{35}{18} \cdot \frac{1}{X_0}=\frac{35m_1+19}{18} \cdot \frac{1}{X_0},
\end{align*}
where we used the trivial bound $T(n_{m_1}) < \tfrac{3}{X_0}$ from Remark~\ref{Remark_T-trvialBound}.

Note that this bound holds for all sums with $m_1$ consecutive local minima.
\end{proof}

A more detailed analysis in the proof of Lemma~\ref{Lemma_T-Abschaetzung} yields a slightly better result:

\begin{corollary}\label{Korrolar_besserTAbschaetzung}
For all $1\leq i\leq m$ we have 
\begin{align*}
T(n_i)&<\frac{97}{54} \cdot \frac{1}{X_0},\\ 
T(n_i)+T(n_{i+1})&<2 \cdot \frac{97}{54} \cdot \frac{1}{X_0}, \\
\text{ or } T(n_i)+T(n_{i+1})+T(n_{i+2})&<3 \cdot \frac{97}{54} \cdot \frac{1}{X_0}.
\end{align*}
\end{corollary}

\begin{proof}
As in the proof of Lemma~\ref{Lemma_T-Abschaetzung} let $k_i$ be the exact number of consecutive o-steps after $n_i$, and let $\ell_i$ be the exact number of consecutive e-steps thereafter. Then we have $C^{k_i+\ell_i}(n_i)=n_{i+1}$.

In the case of $k_i\leq 2$, we get $T(n_i) < \tfrac{5}{3} \cdot \tfrac{1}{n_i} < \tfrac{97}{54} \cdot \tfrac{1}{X_0}$, as with the proof of Lemma~\ref{Lemma_T-Abschaetzung}. Similarly, in the case that $\ell_i\geq 2$, we get $T(n_i) < \tfrac{3}{2} \cdot \tfrac{1}{X_0}<\tfrac{97}{54} \cdot \tfrac{1}{X_0}$. If $\ell_i=1$ and $k_i\geq 4$, we have 
\[T(n_i)+T(n_{i+1}) < \frac{1}{n_i} \cdot \left(3 + 3 \cdot \left(\frac{2}{3}\right)^{k_i}  \right) \leq \frac{1}{n_i} \cdot \frac{97}{27} < 2 \cdot \frac{97}{54} \cdot \frac{1}{X_0},\]
where the second-to-last inequality follows from $3 + 3 \cdot \left(\tfrac{2}{3}\right)^{k_i}\leq 3+ 3\cdot \frac{16}{81}=\frac{97}{27}$ as $k_i\geq 4$.

Thus, the only case that requires closer analysis is $k_i=3$ and $\ell_i=1$. Then there exists a nonnegative integer $a$ with $n_i=a\cdot 16 + 7$ and $n_{i+1}=27 \cdot a + 13$. If $\ell_{i+1}\geq 2$, then---as in the proof of Lemma~\ref{Lemma_T-Abschaetzung}---there exists a nonnegative integer~$n_{i+1}^{\prime}=\tfrac{n_{i+1}-1}{2}$ for which the Collatz sequence merges with the one of $n_{i+1}$ and therefore with the considered $m$-cycle. Thus, we have $n_{i+1}^{\prime}\geq X_0+1$ and subsequently $a \geq \tfrac{2X_0-10}{27}$ and $n_i \geq \tfrac{32X_0-160+189}{27} > \tfrac{32}{27} \cdot X_0$. That leads to the following
bounds:
\begin{align*}
T(n_i) &< \frac{1}{n_i} \cdot \left(1+\frac{2}{3}+\frac{4}{9}\right) =\frac{1}{n_i} \cdot \frac{19}{9}\\
\quad &< \frac{1}{\tfrac{32}{27} \cdot X_0} \cdot \frac{19}{9}= \frac{57}{32} \cdot \frac{1}{X_0} < \frac{97}{54} \cdot \frac{1}{X_0}.
\end{align*}
Hence, the only remaining case is $k_i=3$ and $\ell_i=\ell_{i+1}=1$. As $k_i=3$ and $\ell_i=1$, we know that $T(n_i)<\tfrac{19}{9} \cdot \tfrac{1}{n_i}$ and  $n_{i+1}>\tfrac{27}{16} \cdot n_i$ and that therefore $\tfrac{1}{n_{i+1}} < \tfrac{16}{27} \cdot \tfrac{1}{n_i}$. 

If $k_{i+1}\leq 2$ we have 
\begin{align*}
T(n_{i+1})&<\frac{5}{3} \cdot \frac{1}{n_{i+1}} < \frac{80}{81} \cdot \frac{1}{n_i} < \frac{1}{n_i},\\ 
T(n_i)+T(n_{i+1})&< \frac{28}{9} \cdot \frac{1}{n_i} < 2 \cdot \frac{97}{54} \cdot \frac{1}{X_0}.
\end{align*}

If $k_{i+1}\geq 3$, from the proof of Lemma~\ref{Lemma_T-Abschaetzung} we know
\begin{align*}
T(n_{i+1})+T(n_{i+2}) &< \frac{1}{n_{i+1}} \cdot \frac{35}{9} < \frac{35}{9} \cdot \frac{16}{27} \cdot \frac{1}{n_i}\\
\quad &= \frac{560}{243} \cdot \frac{1}{n_i},\\
T(n_i)+ T(n_{i+1})+T(n_{i+2}) &< \left(\frac{19}{9}+\frac{560}{243}\right) \cdot \frac{1}{n_i} =\frac{1073}{243} \cdot \frac{1}{n_i}\\
\quad &< 3 \cdot \frac{97}{54} \cdot \frac{1}{X_0}.
\end{align*}
\end{proof}

With this, we also get a similar but slightly better result than with Lemma~\ref{LemmaT_Abschaetzung}:

\begin{theorem}\label{Theorem_bessereT_Summe}
Let $0\leq m_1$ be an integer. Then
\[ \sum_{i=1}^{m_1} T(n_i)<\frac{97m_1+73}{54} \cdot \frac{1}{X_0}.\]
\end{theorem}

\begin{proof}
As in the proof of Lemma~\ref{LemmaT_Abschaetzung}, we break the sum into parts of length one, two, or three consecutive summands in the sense from Corollary~\ref{Korrolar_besserTAbschaetzung}. With exception of the last one (in some cases), all of these partial sums are complete. So we know that these partial sums are smaller than $\tfrac{97}{54} \cdot \tfrac{1}{X_0}$ times the number of their respective summands. 

If $T(n_{m_1})$ is part of such a complete partial sum, we have 
\[ \sum_{i=1}^{m_1} T(n_i)<\frac{97m_1}{54} \cdot \frac{1}{X_0}.\]
If $T(n_{m_1})$ is not part of such a complete partial sum but $T(n_{m_1-1})$ is, we have 
\[ \sum_{i=1}^{m_1} T(n_i)<\frac{97(m_1-1)}{54} \cdot \frac{1}{X_0} + \frac{3}{X_0}=\frac{97m_1+65}{54} \cdot \frac{1}{X_0}.\]
Here, we used the trivial bound $T(n_{m_1})<\tfrac{3}{X_0}$, as noted in Remark~\ref{Remark_T-trvialBound}.

If neither $T(n_{m_1})$ nor $T(n_{m_1-1})$ are part of such a complete partial sum, we know that the last complete partial sum ended with $T(n_{m_1-2})$. With this, and by applying Lemma~\ref{LemmaT_Abschaetzung} to the sum 
 $T(n_{m_1-1})+T(n_{m_1})$, we get
\[ \sum_{i=1}^{m_1} T(n_i)<\frac{97(m_1-2)}{54} \cdot \frac{1}{X_0} + \frac{89}{18} \cdot \frac{1}{X_0}=\frac{97m_1+73}{54} \cdot \frac{1}{X_0}.\]
\end{proof}

\section{On Collatz $m$-cycles}\label{Section4}

\begin{definition}
From now on, let $\delta:=\frac{\log(3)}{\log(2)}>1$.
\end{definition}

The following theorem and its proof are similar to the work done by Eliahou in~\cite{SE}.

\begin{theorem}\label{Theorem11}
Let $K$ be the number of odd numbers in a given $m$-cycle and $L$ be the number of even numbers in it. Further let $k_i, n_i$ and $T(n_i)$  be as in Definition~\ref{Definition_k_i}. Then,
\[\delta<\frac{K+L}{K}<\delta+\frac{1}{K \cdot 3 \cdot \log(2)} \cdot \sum_{i=1}^m T(n_i).\]
\end{theorem}

\begin{proof}
Let $\Omega$ be the set of numbers in the $m$-cycle, $\Omega_o$ the subset of odd numbers and $\Omega_e$ the subset of even numbers. Then we have $\left|\Omega_o\right|=K$,  $\left|\Omega_e\right|=L$, and (due to cyclicity)
\begin{align*}
\prod_{n \in \Omega} n &= \prod_{n \in \Omega} C(n) = \prod_{n \in \Omega_e} C(n) \cdot \prod_{n \in \Omega_o} C(n) = \prod_{n \in \Omega_e} \frac{n}{2} \cdot \prod_{n \in \Omega_o} \frac{3n+1}{2}\\
\quad &= \prod_{n \in \Omega_e} \left(n \cdot \frac{1}{2}\right) \cdot \prod_{n \in \Omega_o} \left(n \cdot \frac{1}{2} \cdot \left(3+\frac{1}{n}\right)\right) \\
\quad &= 2^{-(L+K)} \cdot \prod_{n \in \Omega} n  \cdot \prod_{n \in \Omega_o} \left(3+\frac{1}{n}\right),\\
2^{K+L} &= \prod_{n \in \Omega_o} \left(3+\frac{1}{n}\right).
\end{align*}

The right-hand side is obviously larger than $\prod_{n \in \Omega_o} 3 = 3^K$; however, because of the inequality between arithmetic and geometric mean, it is also smaller than 
\begin{align*}
\prod_{n \in \Omega_o} \left(3+\frac{1}{n}\right) &= \left(\sqrt[K]{\prod_{n \in \Omega_o} \left(3+\frac{1}{n}\right)}\right)^K \leq \left(\frac{1}{K} \cdot \sum_{n \in \Omega_o} \left(3+\frac{1}{n}\right) \right)^K\\
\quad &= (3+\mu)^K, 
\end{align*}
where $\mu$ is $\mu:=\tfrac{1}{K} \cdot  \sum_{n \in \Omega_o} \tfrac{1}{n}$. Putting things together, we get
\begin{align*}
3^K & < 2^{K+L} \leq (3+\mu)^K,\\
K \cdot \log (3) &< (K+L) \cdot \log (2) \leq K \cdot \log (3+\mu)\\
\quad &=K \cdot \left(\log(3) + \log\left(1+\frac{\mu}{3}\right)\right)\\
&< K \cdot \left(\log(3) +\frac{\mu}{3} \right),\\
\delta=\frac{\log(3)}{\log(2)}&<\frac{K+L}{K}<\frac{\log(3)}{\log(2)}+\frac{\mu}{3\cdot \log(2)}.
\end{align*}

Now let us rewrite $\mu$:
\[
\mu=\frac{1}{K} \cdot  \sum_{n \in \Omega_o} \frac{1}{n}=\frac{1}{K} \cdot \sum_{i=1}^m \sum_{t=0}^{k_i-1} \frac{1}{C^t(n_i)},
\]
where $k_i$ is the exact number of consecutive o-steps after $n_i$ to the next local maximum on the cycle. Thus, we have $\sum_{i=1}^m k_i=K$, as every such step follows an odd number on the cycle. The inner sum is $T(n_i)$ as defined in the theorem. Hence, we get
\begin{align*}
\mu &= \frac{1}{K} \cdot \sum_{i=1}^m T(n_i),\\
\delta&<\frac{K+L}{K}<\delta+\frac{1}{K\cdot 3 \cdot \log(2)} \cdot  \sum_{i=1}^m T(n_i).
\end{align*}
\end{proof}

This leads to the following result:
\begin{corollary}\label{Corollary_TSumme_ohneRestglied}
With $K, L$ as in Theorem~\ref{Theorem11}, we have 
\[\delta <\frac{K+L}{K}<\delta+\frac{1}{K \cdot 3 \log(2)} \cdot \frac{97}{54} \cdot m \cdot \frac{1}{X_0}.\]
\end{corollary}

\begin{proof}
From Theorem~\ref{Theorem11} we know
\[\delta<\frac{K+L}{K}<\delta+\frac{1}{K \cdot 3\cdot \log(2)} \cdot \sum_{i=1}^m T(n_i).\]
Thus, we need to prove $\sum_{i=1}^m T(n_i)< \tfrac{97}{54} \cdot m \cdot \tfrac{1}{X_0}$. To do so, let $t$ be an arbitrarily large integer and remember that we set $n_{m+1}:=n_1, n_{m+2}:=n_2$, and so on. With $m_1:=t\cdot m$ and using Theorem~\ref{Theorem_bessereT_Summe}, we get
\begin{align*}
t\cdot \sum_{i=1}^m T(n_i) &= \sum_{i=1}^{t\cdot m} T(n_i) < \frac{97 (t\cdot m) +73}{54} \cdot \frac{1}{X_0},\\
\sum_{i=1}^m T(n_i) &< \frac{97}{54} \cdot m \cdot \frac{1}{X_0} + \frac{1}{t} \cdot \frac{73}{54} \cdot \frac{1}{X_0}.
\end{align*}
As one can choose $t$ as large as needed,  the second summand can be lessened to $\varepsilon$ for every $\varepsilon>0$. Thus, we have
\begin{align*}
\sum_{i=1}^m T(n_i) &< \frac{97}{54} \cdot m \cdot \frac{1}{X_0} + \varepsilon \text{ for every $\varepsilon>0$ and, therefore,}\\
\sum_{i=1}^m T(n_i) &\leq  \frac{97}{54} \cdot m \cdot \frac{1}{X_0}, 
\end{align*}
which proves the claim.
\end{proof}

\begin{remark}
Since $\tfrac{1}{3} \cdot \tfrac{97}{54}\approx 0.599$, this is an improvement of more than $40$\% over Corollary~5 of Simons and de Weger in~\cite{SW-www}.
\end{remark}

Using a computer program and the known results for $X_0$, we get a stronger result:
\begin{corollary}\label{Computer_estimate_m_T}
Provided that $X_0\geq 704\cdot 2^{60}$, we have
\[\delta <\frac{K+L}{K}<\delta+\frac{1}{K \cdot 3 \log(2)} \cdot m \cdot \frac{1}{704\cdot 2^{60}}.\]
\end{corollary}

\begin{proof}
In Lemma~\ref{Lemma_T-Abschaetzung} and Corollary~\ref{Korrolar_besserTAbschaetzung}, we proved that the average summand in a complete partial sum $\sum_i T(n_i)$ is smaller than a constant times $\frac{1}{X_0}$. We obtained the best value for this constant in Corollary~\ref{Korrolar_besserTAbschaetzung} with $\frac{97}{54}$. To achieve this result, we had to consider partial sums of up to three consecutive summands in this sum.

If we want to get even lower bounds for this constant, we have two possible options: First, we can increase the number of considered consecutive summands $T(n_i)$ in a partial sum. But this increases also the amount of work necessary in the case-by-case analysis. A second possibility is to use the exact known value of $X_0$ instead of using it as a parameter. This significantly reduces the number of cases to be considered.

Providing that $X_0>704\cdot 2^{60}$ we want to prove that the average $T(n_i)$ is smaller than $\frac{1}{704\cdot 2^{60}}$. To do so, we apply the same methods as above to get upper bounds of the form $c \cdot \tfrac{1}{n_i}$ for sums of $T(n_i)$ for consecutive local minima $n_i$. But if~$r$ is the smallest positive integer with $r\equiv n_i \pmod{2^k}$ for some value of $k$ (the number of, for this case, ``fixed'' steps in the Collatz sequence after $n_i$), we clearly have $n_i\geq r$. If $r$ is large enough, the calculated bound falls below the value of $\tfrac{1}{704\cdot 2^{60}}$. Then we can mark this case as proven.

Since this kind of proof gets very lengthy, we automated this process
by writing a small program in {\tt C++}. It considers partial sums with up to
60 consecutive summands $T(n_i)$.

Having used the program to prove this property, we can proceed as we did in the proof of Corollary~\ref{Corollary_TSumme_ohneRestglied} and obtain the desired result.
\end{proof}

The following lemma also can be found in \cite{SW-www}.

\begin{lemma}\label{Lemma14}
Let $n_i, n_{i+1}$ be two successive local minima in an $m$-cycle. Then we have $n_{i+1} < n_i^{\delta}$.
\end{lemma}

\begin{proof}
Let $k_i$ be the exact number of consecutive o-steps after $n_i$ until the next local maximum in the cycle. By Lemma~\ref{Lemma12} and its proof, we know that there is a positive integer $a$ with $n_i=a\cdot 2^{k_i}-1$ and $C^{k_i}(n_i)=a\cdot 3^{k_i}-1$. Since $k_i$ is the exact number of consecutive such steps after $n_i$, we know that $C^{k_i}(n_i)$ has to be even. Therefore, we know that $n_{i+1}$, the next odd number, fulfills the inequality $n_{i+1}\leq \frac{1}{2} \cdot C^{k_i}(n_i)$.  Thus, we get
\begin{align*}
n_{i+1}&\leq  \frac{1}{2} \cdot C^{k_i}(n_i) = \frac{1}{2} \cdot (a\cdot 3^{k_i}-1) < \frac{3^{k_i}}{2} \cdot a= \frac{3^{k_i}}{2} \cdot \frac{n_i+1}{2^{k_i}}\\
\quad &= \left(2^{k_i}\right)^{\delta-1} \cdot \frac{1}{2} \cdot (n_i+1) \leq \left(a \cdot 2^{k_i}\right)^{\delta-1} \cdot \frac{1}{2} \cdot (n_i+1) = \frac{1}{2} \cdot (n_i+1)^{\delta} \\
\quad &= \frac{1}{2} \cdot \left(\frac{n+1}{n}\right)^{\delta} \cdot n_i^{\delta}<\frac{1}{2} \cdot \left(1+\frac{1}{n}\right)^2 \cdot n_i^{\delta}<n_i^{\delta}, \text{ as $n_i>X_0>3$.}
\end{align*}

Note that the statement is also true for $i=m$ and $n_{m+1}:=n_1$.
\end{proof}

\begin{theorem}\label{Theorem15}
Let $K$ be the number of odd numbers in a given $m$-cycle, and let $L$ be the number of even numbers in it.  Further, assume there exists a positive integer $m_2$  with 
\[m_2\leq m \text{ and } \frac{\delta^{m_2}-1}{\delta-1} \cdot \frac{\log\left(\frac{162}{97} \cdot X_0\right)}{\log(2)} \leq \frac{m_2}{m} \cdot K.\]
Let $v=\tfrac{m_2}{m} \cdot K \cdot \tfrac{\delta-1}{\delta^{m_2}-1}$. Then we get 
\[\delta <\frac{K+L}{K}<\delta+\frac{1}{K \cdot 3 \log(2)} \cdot \left( \frac{97(m-m_2)+73}{54 \cdot X_0}+\frac{3}{2^v-1}+ \frac{3 \cdot (m_2-1)}{(2^v-1)^{\delta}} \right).\]
If $m_2=m-1$ we get the better estimate
\[\delta <\frac{K+L}{K}<\delta+\frac{1}{K \cdot 3 \log(2)} \cdot \left( \frac{3}{X_0}+\frac{3}{2^v-1}+ \frac{3 \cdot (m-2)}{(2^v-1)^{\delta}} \right)\]
and for $m_2=m$
\[\delta <\frac{K+L}{K}<\delta+\frac{1}{K \cdot 3 \log(2)} \cdot \left( \frac{3}{2^v-1}+ \frac{3 \cdot (m-1)}{(2^v-1)^{\delta}} \right).\]
\end{theorem}

\begin{proof}
First, let $m_2$ be as given in this theorem. If the only nonnegative integer fulfilling the given inequality is $m_2=0$, Corollary~\ref{Corollary_TSumme_ohneRestglied} and Corollary~\ref{Computer_estimate_m_T} give good upper bounds for $\tfrac{K+L}{K}$. Now let the inequality be true for some integer $m_2>0$.

From Theorem~\ref{Theorem11}, we know
\[\delta<\frac{K+L}{K}<\delta+\frac{1}{K \cdot 3\cdot \log(2)} \cdot \sum_{i=1}^m T(n_i).\]
Therefore, we have to give upper bounds on $\sum_{i=1}^m T(n_i)$ in the different cases. To do so, we use Theorem~\ref{Theorem_bessereT_Summe} and Lemma~\ref{Lemma14}:

There are $m_2$ consecutive local minima which are directly followed by at least $\frac{m_2}{m} \cdot K$ o-steps in total. (If this would not be the case, all of the sums $k_1+\cdots+k_{m_2}, k_2+\cdots+k_{m_2+1}, \ldots$ and $k_{m}+k_1+\cdots+k_{m_2-1}$ would be smaller than $\tfrac{m_2}{m} \cdot K$. But adding these $m$~sums together gives $m_2 \cdot K$, as every $k_i$ occurs in exactly $m_2$ such sums and $\sum_{i=1}^{m} k_i=K$, which contradicts the assumption.) W.l.o.g., let these $m_2$ consecutive local minima be $n_{m-m_2+1}, n_{m-m_2+2}, \ldots, n_m$.

From Lemma~\ref{Lemma12} we know that $k_i\leq \tfrac{\log(n_i+1)}{\log(2)}$. Now Lemma~\ref{Lemma14} gives $n_{i+1}<n_i^{\delta}$ and, therefore, $n_{i+1}+1<n_i^{\delta}+1<(n_i+1)^{\delta}$ which leads to the 
inequality $k_{i+1}<\delta \cdot \tfrac{\log(n_i+1)}{\log(2)}$. Subsequently, in an analogous way we get $k_{i+\ell} < \delta^{\ell} \cdot \tfrac{\log(n_i+1)}{\log(2)}$.

Now we can sum up these estimates for $i=m-m_2+1, \ldots, m$:
\begin{align*}
\frac{m_2}{m} \cdot K &\leq  \sum_{i=m-m_2+1}^{m} k_i = \sum_{\ell=0}^{m_2-1} k_{(m-m_2+1)+\ell}\\
\quad &\leq \frac{\log(n_{m-m_2+1}+1)}{\log(2)} \cdot \sum_{\ell=0}^{m_2-1} \delta^{\ell}\\
\quad &= \frac{\log(n_{m-m_2+1}+1)}{\log(2)} \cdot \frac{\delta^{m_2}-1}{\delta-1}.
\end{align*}
From the premise on $m_2$, we know that this is larger than 
\begin{align*}
\quad &\phantom{=} \frac{\delta^{m_2}-1}{\delta-1} \cdot \frac{\log\left(\frac{162}{97} \cdot X_0\right)}{\log(2)}.
\end{align*}
Thus, we get
\[n_{m-m_2+1}+1 > \frac{162}{97} \cdot X_0 \text{ and, therefore, } n_{m-m_2+1}\geq \frac{162}{97} \cdot X_0.\]

Now with Remark~\ref{Remark_T-trvialBound} we conclude $T(n_{m-m_2+1})<\tfrac{3}{n_{m-m_2+1}}<\tfrac{97}{54}\cdot \tfrac{1}{X_0}$. Thus, this estimate on $T(n_{m-m_2+1})$ is at least as good as the one we get from Theorem~\ref{Theorem_bessereT_Summe}. But we can get a better one:

From the definition of~$v$, we get  $v\cdot \tfrac{\delta^{m_2}-1}{\delta-1}=\tfrac{m_2}{m} \cdot K$. Thus, we have 
\[
v \leq \frac{\log(n_{m-m_2+1}+1)}{\log(2)} \text{ or, equivalently, } 2^v-1 \leq n_{m-m_2+1}.
\]
Now this gives $T(n_{m-m_2+1})<\tfrac{3}{2^v-1}$ and, since 
\[n_{m-m_2+1+\ell}<n_{m-m_2+1}^{(\delta^{\ell})}\leq n_{m-m_2+1}^{\delta},\] 
we get the inequality $T(n_{m-m_2+1+\ell})<\frac{3}{(2^v-1)^{\delta}}$.
Using this together with Theorem~\ref{Theorem_bessereT_Summe} and $m_1:=m-m_2$, we get
\begin{align*}
\sum_{i=1}^m T(n_i) &= \sum_{i=1}^{m-m_2} T(n_i) + T(n_{m-m_2+1}) + \sum_{i=m-m_2+1+1}^m T(n_i)\\
\quad &< \frac{97(m-m_2)+1}{54\cdot X_0} + \frac{3}{2^v-1} + \frac{3 \cdot (m_2-1)}{(2^v-1)^{\delta}},
\end{align*}
which proves the general statement for $1\leq m_2\leq m$. For $m_2=m-1$, the first sum $\sum_{i=1}^{m-m_2} T(n_i)=T(n_1)$ can be better bounded above by $\tfrac{3}{X_0}$, as seen in Remark~\ref{Remark_T-trvialBound}. For $m_2=m$, this first sum is empty and therefore zero.
\end{proof}

We now give a well known lemma for diophantine approximation:

\begin{lemma}\label{Lemma16}
Let $0<\alpha<\beta$ be two real numbers with continued fraction expansions $\alpha=[a_0;a_1,\ldots,a_{k-1},a_k,\ldots]$ and $\beta=[a_0;a_1,\ldots,a_{k-1},b_k,\ldots]$.
Then every fraction in the open interval $(\alpha,\beta)$ has a denominator which is not smaller than the one of $\gamma=[a_0;a_1,\ldots,a_{k-1},c_k]$ with $c_k=\min(a_k,b_k)+1$.
\end{lemma}

\begin{proof}
First, let $\gamma$ be some real number with $\alpha<\gamma<\beta$. Since $\floor{\alpha}=a_0=\floor{\beta}$, we also have $\floor{\gamma}=a_0$. Thus, $0<\alpha-a_0<\gamma-a_0<\beta-a_0<1$ and, therefore, $1<\tfrac{1}{\beta-a_0}<\tfrac{1}{\gamma-a_0}<\tfrac{1}{\alpha-a_0}$. Since $\floor{\tfrac{1}{\beta-a_0}}=a_1=\floor{\tfrac{1}{\alpha-a_0}}$, we also get $\floor{\tfrac{1}{\gamma-a_0}}=a_1$, and so on.

Thus, for every real number $\gamma$ within the interval $(\alpha,\beta)$, we get the same beginning in the continued fraction expansion of $\gamma$: $\gamma=[a_0;a_1,\ldots,a_{k-1},\ldots]$. 

If $k$ is even, increasing the partial denominator increases the fraction. Thus, we have $a_k< b_k$ and  \[[a_0;a_1,\ldots,a_{k-1},a_k] < \alpha < [a_0;a_1,\ldots,a_{k-1},a_k+1] < \beta.\] If we choose $c_k:=a_k+1=\min(a_k+b_k)+1$, we get the fraction with smallest denominator in this interval.

If $k$ is odd, decreasing the partial denominator increases the fraction. Thus, we have $a_k>b_k$ and \[\alpha < [a_0;a_1,\ldots,a_{k-1},b_k+1] < \beta < [a_0;a_1,\ldots,a_{k-1},b_k].\]  If we choose $c_k:=b_k+1=\min(a_k+b_k)+1$, we get the fraction with smallest denominator in this interval.
\end{proof}

\begin{theorem}[Main Theorem]\label{Theorem17}
There is no $m$-cycle with $m\leq 91$.
\end{theorem}

\begin{proof}
Let $m\leq 91$. From Theorem~3 in \cite{SW-www} we know that $K>7\cdot 10^{11}$. 

With $X_0=704\cdot 2^{60}\approx 8.1\cdot10^{20}$, as given in Definition~\ref{Definition8}, we get $m_2\geq 47$. Now, from Theorem~\ref{Theorem15} we have
\[\delta < \frac{K+L}{K} < \delta + 6.9\cdot 10^{-32}.\]
By using continued fractions as in Lemma~\ref{Lemma16}, we can now get a better lower bound for $K$ with
\[K>5.2\cdot 10^{15}.\]
And now we can reiterate this process! 

In the next run, we get $m_2\geq 67, \delta<\tfrac{K+L}{K}<\delta+5.1\cdot 10^{-36}$ and, therefore, $K>3.97\cdot 10^{17}$.

Applying this newer bound on $K$ once more, we get $m_2\geq 77$, \linebreak $\delta<\tfrac{K+L}{K}<\delta+4.1\cdot 10^{-38}$ and, hence, $K>4.64\cdot 10^{18}$. 

In the forth run we get $m_2\geq 82, \delta<\tfrac{K+L}{K}<\delta+2.3\cdot 10^{-39}$ and with this $K>2.74\cdot 10^{19}$. From there we get $m_2\geq 86, \delta<\tfrac{K+L}{K}<\delta+2.3\cdot 10^{-40}$ and $K>7.76\cdot 10^{19}$.  And from this we get $m_2\geq 88, \delta<\tfrac{K+L}{K}<\delta+5.3\cdot 10^{-41}$ and  $K>2.05\cdot 10^{20}$.

Applying this lower bounnd a last time we get $m_2\geq 91, \delta<\tfrac{K+L}{K}<\delta+1.11\cdot 10^{-43}$ and  $K>7.94\cdot 10^{21}$.

But this last lower bound on $K$ is larger than the upper bound of $K<1.4784 m\delta^m<2.2\cdot10^{20}$ given by Simons and de Weger~\cite{SW-www}. Thus, no such $m$-cycle can exist.
\end{proof}

Using the same technique as above and Corollaries~\ref{Corollary_TSumme_ohneRestglied} and~\ref{Computer_estimate_m_T}, we also get new lower bounds on $K$ for $m$-cycles with $m\geq 92$:

\begin{corollary}\label{Corollary_Computer_greater_m_estimate}
In Table~\ref{Collatz_Table_One} different pairs of values $m$ und $K$ are listed. If there is a $m$-cycle with $m$ equal or smaller than the given value, this cycle consists of at least the corresponding number of $K$ odd members.

\begin{table}[htb]
\begin{center}
\begin{tabular}{c|c}
$m$ & $K$\\
\hline
$98$ & $7.76\cdot 10^{19}$ \\
$117$ & $2.74\cdot 10^{19}$\\
$369$ & $4.64\cdot 10^{18}$\\
$4366$ & $3.97\cdot 10^{17}$\\
$17096$ & $1.30\cdot 10^{17}$\\
$\vdots$ & $\vdots$\\
$802380$ & $5.26\cdot 10^{15}$\\
$1.07\cdot 10^6$ & $4.78\cdot 10^{15}$\\
$\vdots$ & $\vdots$\\
$1.89\cdot 10^9$ & $1.64\cdot 10^{12}$\\
$2.18\cdot 10^9$ & $8.90\cdot 10^{11}$\\
$1.34\cdot 10^{10}$ & $1.37\cdot 10^{11}$\\
For all $m\in\mathbb{N}$ & $7.20\cdot 10^{10}$
\end{tabular}
\caption{Corollary~\ref{Corollary_Computer_greater_m_estimate}: If $m\leq \cdots$, then $K>\cdots$.}
\label{Collatz_Table_One}
\end{center}
\end{table}
\end{corollary}

\begin{proof}
The first two lines follow from Theorem~\ref{Theorem15}, all others except the last one from Corollary~\ref{Computer_estimate_m_T}, and the last one from the statement independent of~$m$ in \cite{He,Pu} or \cite{SE}, using the known value of $X_0$ as in Definition~\ref{Definition8} and given in \cite{Barina-www}.
\end{proof}

\begin{remark}
Over time, with better known lower limits on $X_0$, one gets slightly better results in Corollary~\ref{Corollary_Computer_greater_m_estimate}. The results in this corollary are in some cases significant improvements on the bounds given by  Theorem~3 in \cite{SW-www}.
\end{remark}

\section{Cycles without knowing $m$}\label{section5}
In a manner similar to that in Lemma~\ref{Lemma_T-Abschaetzung} and Corollary~\ref{Korrolar_besserTAbschaetzung}, we want to give an upper bound on another weighted average of the $T(n_i)$:

\begin{lemma}\label{Lemma_T-Abschaetzung_ohne_m}
For all $1\leq i$, we have 
\begin{align*}
T(n_i) & < k_i \cdot \frac{3}{4} \cdot\frac{1}{X_0},\\
T(n_i) + T(n_{i+1}) & < (k_i+k_{i+1}) \cdot \frac{3}{4} \cdot\frac{1}{X_0}, \\
\text{ or } T(n_i) + T(n_{i+1})  + T(n_{i+2}) & < (k_i+k_{i+1}+k_{i+2}) \cdot \frac{3}{4} \cdot\frac{1}{X_0}.
\end{align*}
\end{lemma}

\begin{proof}
As in the proofs of Lemma~\ref{Lemma_T-Abschaetzung} and Corollary~\ref{Korrolar_besserTAbschaetzung}, let $k_i$ be the exact number of o-steps following $n_i$ and $\ell_i$ the exact number of e-steps following them, such that we have $n_{i+1}=C^{k_i+\ell_i}(n_i)$.

If $k_i=1$, there exists a nonnegative integer $a$ with $n_i=4\cdot a + 1$ such that $C^2(n_i)=3\cdot a +1$. Since $C^2(n_i)\geq X_0+1$, we have $n_i\geq \tfrac{4}{3} \cdot X_0 + 1>\tfrac{4}{3} \cdot X_0$ and, therefore, $T(n_i)<\tfrac{3}{4} \cdot \tfrac{1}{X_0}$.

If $k_i=3$ ,we know $T(n_i)<\tfrac{19}{9} \cdot \tfrac{1}{X_0}<3 \cdot \tfrac{3}{4} \cdot \tfrac{1}{X_0}$.

If $k_i\geq 4$, we have $T(n_i)<3 \cdot \tfrac{1}{X_0} \leq k_i \cdot \tfrac{3}{4} \cdot \tfrac{1}{X_0}$.

This only leaves the case $k_i=2$ for further investigation. There we have $T(n_i)=\tfrac{5}{3} \cdot\tfrac{1}{n_i}$. If $\ell_i\geq 2$, we know from the cited proofs that $n_i>2X_0$ and, therefore, $T(n_i) <\tfrac{5}{6} \cdot\tfrac{1}{X_0}<2 \cdot \tfrac{3}{4} \cdot \tfrac{1}{X_0}$. 

Thus, let $k_i=2$ and $\ell_i=1$. Then there exists a nonnegative integer $a$ with $n_i=8\cdot a + 3$ and $n_{i+1}=9 \cdot a + 4$. If $k_{i+1}=1$, we have $n_{i+1}\geq \tfrac{4}{3} \cdot X_0 + 1$ and $n_i\geq \tfrac{32}{27} \cdot X_0 + \tfrac{1}{3}>\tfrac{32}{27} \cdot X_0$. With this, we get
\[T(n_i)< \frac{5}{3} \cdot \frac{27}{32} \cdot \frac{1}{X_0}=\frac{45}{32}\cdot \frac{1}{X_0} < 2 \cdot \frac{3}{4} \cdot\frac{1}{X_0}.\]

In the case $k_i=2$ and $\ell_i=1$ we generally have $n_{i+1} > \tfrac{9}{8} \cdot n_{i}$. Thus, if $k_{i+1}\geq 5$, we get
\begin{align*}
T(n_i)+T(n_{i+1})&<\left( \frac{5}{3} + 3 \cdot \frac{9}{8}\right)\cdot \frac{1}{X_0} = \frac{121}{24} \cdot \frac{1}{X_0}\\
\quad &<7 \cdot \frac{3}{4}\cdot \frac{1}{X_0} \leq (k_i+k_{i+1}) \cdot \frac{3}{4}\cdot \frac{1}{X_0}.
\end{align*}

If $k_i=2, \ell_i=1$ and $k_{i+1}=4$, with the better estimate $T(n_{i+1})<\tfrac{65}{27} \cdot \tfrac{1}{n_{i+1}}$ the following inequality holds:
\begin{align*}
T(n_i)+T(n_{i+1})&<\left( \frac{5}{3} + \frac{65}{27} \cdot \frac{9}{8}\right)\cdot \frac{1}{X_0} = \frac{105}{24} \cdot \frac{1}{X_0}\\
\quad &<6 \cdot \frac{3}{4}\cdot \frac{1}{X_0} = (k_i+k_{i+1}) \cdot \frac{3}{4}\cdot \frac{1}{X_0}.
\end{align*}

This leaves $k_i=2, \ell_i=2$ and $k_{i+1}\in\{2,3\}$. If $\ell_{i+1}\geq 2$, we have $n_{i+1}>2X_0$ as above and, therefore, $n_{i+1}\geq 2X_0+1$ and, hence, $n_i=\tfrac{8}{9} \cdot n_{i+1}-\tfrac{5}{9}>\tfrac{16}{9} \cdot X_0$. With this, we get
\[T(n_i) < \frac{5}{3} \cdot \frac{9}{16} \cdot \frac{1}{X_0} < 2 \cdot \frac{3}{4}  \cdot \frac{1}{X_0}.\]

Therefore, from now on we can assume $\ell_{i+1}=1$. Let $k_i=2, \ell_i=1$ and in this case $k_{i+1}=2$. Then there exists a nonnegative integer~$a$ with $n_i=2^6 \cdot a + 59$ and $C^{2+1+2+1}(n_i)=n_{i+2}=3^4 \cdot a + 76$. Moreover, we have 
\[T(n_i)+T(n_{i+1})< \frac{5}{3} \cdot \frac{1}{n_i}+ \frac{5}{3} \cdot \frac{1}{n_{i+1}}<\frac{5}{3} \cdot \frac{1}{n_i}+ \frac{5}{3} \cdot \frac{8}{9} \cdot \frac{1}{n_{i}}=\frac{85}{27} \cdot \frac{1}{n_i}.\]

If $k_i=2, \ell_i=1, k_{i+1}=2, \ell_{i+1}=1$ and $k_{i+2}=1$, we get $n_{i+2}\geq \tfrac{4}{3} \cdot X_0+1$ and, therefore, $n_i> \tfrac{256}{243} \cdot X_0 -1$. Hence,
\begin{align*}
T(n_i)+T(n_{i+1})&< \frac{85}{27} \cdot \frac{n_i+1}{n_i} \cdot \frac{1}{n_i+1} \\
\quad &< \frac{85}{27} \cdot \frac{243}{256} \cdot \frac{n_i+1}{n_i} \cdot \frac{1}{X_0}\\
\quad &= \frac{765}{256} \cdot \frac{n_i+1}{n_i} \cdot \frac{1}{X_0}\\
\quad &< \frac{765}{256} \cdot \frac{766}{765} \cdot \frac{1}{X_0}\text{ (as $n_i>X_0>765$)}\\
\quad &< \frac{766}{256} \cdot \frac{1}{X_0} < 4\cdot \frac{3}{4} \cdot \frac{1}{X_0}.
\end{align*}

If $k_{i+2}=2$, we get 
\[T(n_{i+2})=\tfrac{5}{3} \cdot \tfrac{1}{n_{i+2}}<\tfrac{5}{3} \cdot \tfrac{64}{81} \cdot \tfrac{1}{X_0}=\tfrac{320}{243} \cdot \tfrac{1}{X_0}\]
and, thus,
\begin{align*}
T(n_i)+T(n_{i+1})+T(n_{i+2})&<\left( \frac{85}{27} + \frac{320}{243}\right)\cdot \frac{1}{X_0} = \frac{1085}{243} \cdot \frac{1}{X_0}\\
\quad &<6 \cdot \frac{3}{4}\cdot \frac{1}{X_0} = (k_i+k_{i+1}+k_{i+2}) \cdot \frac{3}{4}\cdot \frac{1}{X_0}.
\end{align*}

In the same way, for $k_{i+2}=3$ we get $T(n_{i+2})<\tfrac{19}{9} \cdot \tfrac{64}{81} \cdot \tfrac{1}{X_0}=\tfrac{1216}{729} \cdot \tfrac{1}{X_0}$ and
\begin{align*}
T(n_i)+T(n_{i+1})+T(n_{i+2})&<\left( \frac{85}{27} + \frac{1216}{729}\right)\cdot \frac{1}{X_0} = \frac{3511}{729} \cdot \frac{1}{X_0}\\
\quad &<7 \cdot \frac{3}{4}\cdot \frac{1}{X_0} = (k_i+k_{i+1}+k_{i+2}) \cdot \frac{3}{4}\cdot \frac{1}{X_0}.
\end{align*}

For $k_{i+2}\geq 4$ we have $T(n_{i+2})<3 \cdot \tfrac{64}{81} \cdot \tfrac{1}{X_0}=\tfrac{64}{27} \cdot \tfrac{1}{X_0}$. Thus, 
\begin{align*}
T(n_i)+T(n_{i+1})+T(n_{i+2})&<\left( \frac{85}{27} + \frac{64}{27}\right)\cdot \frac{1}{X_0} = \frac{149}{27} \cdot \frac{1}{X_0}\\
\quad &<8 \cdot \frac{3}{4}\cdot \frac{1}{X_0} = (k_i+k_{i+1}+k_{i+2}) \cdot \frac{3}{4}\cdot \frac{1}{X_0},
\end{align*}
which concludes the proof in the case $k_i=2, \ell_i=1, k_{i+1}=2$ and $\ell_{i+1}=1$. 
 
Now we only need to consider the case $k_i=2, \ell_i=1, k_{i+1}=3$ and $\ell_{i+1}=1$. Here, there exists a nonnegative integer~$a$ with $n_i=2^7 \cdot a + 91$ and $C^{2+1+3+1}(n_i)=n_{i+2}=3^5 \cdot a + 175$. We also have
\[T(n_i)+T(n_{i+1})< \frac{5}{3} \cdot \frac{1}{n_i}+ \frac{19}{9} \cdot \frac{1}{n_{i+1}}<\frac{5}{3} \cdot \frac{1}{n_i}+ \frac{19}{9} \cdot \frac{8}{9} \cdot \frac{1}{n_{i}}=\frac{297}{81} \cdot \frac{1}{n_i}.\]
 
If $k_{i+2}=1$, we have $T(n_{i+2})=\tfrac{1}{n_{i+2}}<\tfrac{128}{243} \cdot \tfrac{1}{n_i}$ and, therefore,
\begin{align*}
T(n_i)+T(n_{i+1})+T(n_{i+1})&< \left(\frac{297}{81}+\frac{128}{243}\right) \cdot \frac{1}{X_0}\\
\quad &= \frac{1019}{243}\cdot \frac{1}{X_0} < 6 \cdot \frac{3}{4}\cdot \frac{1}{X_0}\\
\quad &=  (k_i+k_{i+1}+k_{i+2}) \cdot \frac{3}{4}\cdot \frac{1}{X_0}.
\end{align*}

At last, if $k_{i+2}\geq 2$ we have $T(n_{i+2})<\tfrac{3}{n_{i+2}}<\tfrac{128}{81} \cdot \tfrac{1}{n_i}$, and with this, we get
\begin{align*}
T(n_i)+T(n_{i+1})+T(n_{i+1})&< \left(\frac{297}{81}+\frac{128}{81}\right) \cdot \frac{1}{X_0}\\
\quad &= \frac{425}{81}\cdot \frac{1}{X_0} < 7 \cdot \frac{3}{4}\cdot \frac{1}{X_0}\\
\quad &\leq  (k_i+k_{i+1}+k_{i+2}) \cdot \frac{3}{4}\cdot \frac{1}{X_0},
\end{align*}
which concludes the last open case and, therefore, the proof.
\end{proof}

This leads to a result similar to that yielded by the Main Theorem~\ref{Theorem15}, but independent of the number~$m$ of local minima in the nontrivial cycle:

\begin{theorem}\label{Theorem-Abschaetzungohnem}
Let $K$ be the number of odd numbers in a given $m$-cycle, and let $L$ be the number of even numbers in it. Then
\[\delta <\frac{K+L}{K}<\delta+\frac{1}{3 \log(2)} \cdot \frac{3}{4} \cdot \frac{1}{X_0}.\]
\end{theorem}

\begin{proof}
As in the proof of Theorem~\ref{Theorem15}, the only thing we have to do is proving the inequality
\[\sum_{i=1}^m T(n_i) \leq K \cdot \frac{3}{4} \cdot \frac{1}{X_0}.\]
To do so, we proceed as we did in the proofs of Theorem~\ref{Theorem_bessereT_Summe} and Corollary~\ref{Corollary_TSumme_ohneRestglied}. First, let $m_1$ be a positive integer and consider the sum \[\sum_{i=1}^{m_1} T(n_i).\] We break this sum into partial sums consisting of one, two, or three consecutive summands. We can do this in such a way that for all partial sums (with the possible exception of the last) we know by Lemma~\ref{Lemma_T-Abschaetzung_ohne_m} that each of these is smaller than $k \cdot \tfrac{3}{4} \cdot \tfrac{1}{X_0}$, where $k$ is the number of o-steps in the at most three consecutive strictly increasing parts of the given nontrivial cycle. The last partial sum consists of at most two summands. Using the trivial bounds $T(n_{m_1-1})<3\cdot \tfrac{1}{X_0}$ and $T(n_{m_1})<3\cdot \tfrac{1}{X_0}$, we get
\begin{align*}
\sum_{i=1}^{m_1} T(n_i) &< \left(\sum_{i=1}^{m_1} k_i\right) \cdot \frac{3}{4} \cdot \frac{1}{X_0} + 2\cdot 3 \cdot \frac{1}{X_0}\\
\quad &=\left(\sum_{i=1}^{m_1} k_i + 8\right) \cdot \frac{3}{4} \cdot \frac{1}{X_0}.
\end{align*}

Now, let $t$ be an arbitrarily large integer and $m_1=t\cdot m$. Then we get
\begin{align*}
t\cdot \sum_{i=1}^m T(n_i) &= \sum_{i=1}^{t\cdot m} T(n_i) < \left(\sum_{i=1}^{t\cdot m} k_i + 8\right) \cdot \frac{3}{4} \cdot \frac{1}{X_0}\\
\quad &=(t\cdot K + 8)\cdot  \frac{3}{4} \cdot \frac{1}{X_0}, \text{ where we used $\sum_{i=1}^{t\cdot m} k_i=t\cdot \sum_{i=1}^{m} k_i=t\cdot K$.}
\end{align*}
This leads directly to
\begin{align*}
\sum_{i=1}^m T(n_i) &<K\cdot  \frac{3}{4} \cdot \frac{1}{X_0} + \varepsilon \text{ for all $\varepsilon>0$ and, therefore,}\\
\sum_{i=1}^m T(n_i) &\leq K\cdot  \frac{3}{4} \cdot \frac{1}{X_0} \text{ as desired.}
\end{align*}
 \end{proof}

\begin{remark}\label{RemarkAbschaetzungohnem}
Theorem~\ref{Theorem-Abschaetzungohnem} leads to a reduction of the bounds on $X_0$ given in or computed by the methods in~\cite{SE}, which are needed for reaching the next threshold on~$K$ in a nontrivial cycle, by $25$\%. Thus, to prove that every nontrivial cycle contains at least $K> 1.375\cdot 10^{11}$  odd numbers, using previous methods one would have to show $\delta<\tfrac{K+L}{K}<\delta+1.1032\cdot 10^{-22}$, and therefore, $X_0\geq 3781\cdot 2^{60}$. With the result given in Theorem~\ref{Theorem-Abschaetzungohnem}, it suffices to show $X_0\geq 2836\cdot 2^{60}$.
\end{remark}

If one is interested in the problem of lowering the value for $X_0$ given in the last remark to the greatest extend possible, one can use the methods of Lemma~\ref{Lemma_T-Abschaetzung_ohne_m} in an automated way.

\begin{corollary}\label{Corollary_Computer-Bound}
If $X_0\geq 1536\cdot 2^{60}=3\cdot2^{69}$ then every nontrivial cycle contains at least $K> 1.375\cdot 10^{11}$  odd numbers. 
\end{corollary}

\begin{proof}
We have written a small {\tt C++} program, which mainly uses the methods given in the proof of  Lemma~\ref{Lemma_T-Abschaetzung_ohne_m}. There is one key exception: for every situation, the program tracks the residue classes modulo powers of two for a given number $n_i$. If the smallest positive member of this residue class is already larger than the needed bound of $3781\cdot 2^{60}$, this case no longer needs to be considered, since then one can use the trivial bound $T(n_i)\leq \tfrac{k_i}{n_i}<k_i \cdot \tfrac{1}{3781 \cdot 2^{60}}$.

This bound on $X_0$ was reached after five weeks of computing time on an i9~processor of the 11th generation with 8~cores.
\end{proof}

\begin{remark}
If one wants to prove that every nontrivial cycle has at least $K> 1.375\cdot 10^{11}$  odd members with  Corollary~\ref{Corollary_Computer-Bound} it suffices to proof for all numbers less than or equal to $1536\cdot 2^{60}$ that their respective Collatz sequences reach the trivial cycle. This bound on $X_0$ is only about 40\% of the one needed by the methods in~\cite{SE}, thus saving nearly 60\% of computing time requierd to reach the next threshold for the next larger lower bound on the number of members a nontrivial Collatz cycle must have. As the current fastest project for increasing $X_0$, see~\cite{Barina-www}, needed over 27~months for completing the search up to $512\cdot2^{60}=2^{69}$, this reduction has some impact.  

In our numerical experiments, we get that by lowering the bound on $X_0$ by $2^{60}$ (that is, e.g., from $1536\cdot 2^{60}$ to $1535\cdot2^{60}$) the computing time needed increases by about $1.5$\%. Thus, the chosen bound of $X_0$ in Corollary~\ref{Corollary_Computer-Bound} cannot be decreased by a reasonable amount before this computation becomes less effective than the search in \cite{Barina-www}.
\end{remark}

\begin{remark}
All results in this paper, with the exception of Corollaries~\ref{Computer_estimate_m_T}, \ref{Corollary_Computer_greater_m_estimate}, and \ref{Corollary_Computer-Bound}, are independent of the exact value of~$X_0$, provided it is not too small. (In all other lemmas, theorems, and corollaries we needed at most $X_0>765$.) Thus, they scale up and give better absolute values, as better results on $X_0$ become available.
\end{remark}

\bigskip
\hrule
\bigskip

\noindent 
2020 \emph{Mathematics Subject Classification}:~Primary 11B83.

\medskip

\noindent 
\emph{Keywords}:~Collatz conjecture, Syracuse problem, nontrivial Collatz cycle. 
\bigskip
\hrule
\bigskip

\end{document}